\documentclass[11pt]{article}

\title{\huge Steenrod operations on bar complex}

\date{}
\author{Syunji Moriya\footnote{Corresponding address: Department of Mathematics, Faculty of Science, 
Kyoto University, 
Kyoto, 606-8502, Japan.          \   
E-mail adress: \texttt{moriyasy@math.kyoto-u.ac.jp} \ 
Telephone number: 81-075-753-3700 \ 
FAX number: 81-075-753-3711}}
\usepackage{amsthm}
\usepackage{amscd}
\usepackage{amsmath,amsthm,amssymb,amscd,mathrsfs}
\usepackage[arrow,matrix]{xy}
\input xy
\xyoption{all}

\usepackage{enumerate}
\theoremstyle{plain}
\newtheorem{defi}{Definition}[subsection]

\newtheorem{prop}[defi]{Proposition}

\newtheorem{rem}[defi]{Remark}

\newtheorem{lem}[defi]{Lemma}

\newtheorem{thm}[defi]{Theorem}

\newtheorem{exa}[defi]{Example}

\newcommand{\Z}{\mathbb{Z}}
\newcommand{\E}{\mathcal{E}}
\newcommand{\oper}{\mathcal{O}}
\newcommand{\A}{\mathcal{A}}
\newcommand{\ol}{\overline}

\newcommand{\mbf}{\mathbf}
\newcommand{\base}{\mathbf{k}}
\begin{document}
\maketitle
\begin{abstract}
We define a chain map of the form $\E(k)\otimes BA^{\otimes k}\longrightarrow BA$ , where $\E$ is a combinatorial $E_\infty$-operad called the sequence operad, and $BA$ is the bar complex of an $\E$-algebra $A$. We see that Steenrod-type operations derived from the chain map are equal to the corresponding operations on the cohomology of the based loop space under an isomorphism.
\end{abstract}
\section{Introduction}
 The bar complex is a model of cochain of a based loop space. J.R.Smith \cite{smith} and B.Fresse \cite{fresse1} enriched the bar complex of an $E_\infty$-algebra
with an $E_\infty$-structure. In view of M.A.Mandell's theorem \cite[Main Theorem]{mandell}, this enrichment provides a complete algebraic model of $p$-adic homotopy type of a based loop space and enables them to iterate the bar construction. As another complete model, the categorical bar complex is known but the (classical) bar complex has the advantage that one does not need to take a cofibrant replacement of an algebra to obtain the right answer with it.  Fresse defined   the $E_\infty$-structure, using systematically bar modules and model category structures on the categories of modules and algebras over an operad. \\ 
\indent Our motivation is to find a combinatorial alternative of these $E_\infty$-structures on the bar complex. In this article,  we prove  the following theorem.
\begin{thm}[Thm.\ref{thmchain}]\label{mainthm}
Let $k$ be a positive integer. Let $\E$ denote the sequence operad and $A$ be an $\E$-algebra. Let $BA$ be the bar complex of $A$. There exists a chain map
\[
\Phi_k:\E(k)\otimes_{\Sigma_k}BA^{\otimes k}\longrightarrow BA.
\]
Here, $\Sigma_k$ is the $k$-th symmetric group.
\end{thm}
Here, the sequence operad is a small combinatorial model of $E_\infty$-operad introduced by J.E.McClure and J.H.Smith \cite{mccluresmith}, which naturally acts on the normalized cochain. 
We warn the reader that the chain map $\Phi_k$ does \textbf{not} define an operad action. So this result cannot be used for iteration.\\
\indent The utility of the map $\Phi_k$ is that Steenrod operations can be derived from it. In fact, we can apply the framework of J.P.May \cite{may} to the map $\Phi_k$ in order to define operations $P^s$ and $\beta P^s$.  We show these operations are isomorphic to the corresponding operations on the cohomology of a based loop space by a simple application of an argument of Fresse (see section \ref{steenrod}). \\
\indent We shall mention preceeding works. H.J.Baues \cite{baues} defined a product on 
the bar complex of normalized cochains which is equivalent to the cup product.  T.Kadeishvili \cite{kadeishvili} defined $\cup_i$-product on the bar complex, generalizing Baues' construction (over a field of characteristic 2). Thm.        \ref{mainthm} is considered as a generalization of them as the operations of Baues and Kadeishvili are equal to evaluations of $\Phi_2$ at some elements, see Prop.\ref{propcomparison}. For another generalization, see Fresse \cite{fresse2}. \\
\indent In the last section, we define a diagonal on the sequence operad. An immediate consequence is that the tensor of two $\E$-algebra has functorial $\E$-algebra structure. For Barrat-Eccles operad \cite{bergerfresse}, which is another combinatorial model of $E_\infty$-operad, a diagonal is already known. As the action of the sequence operad on the normalized cochain is transparent, the tensor product provides a simple model of  product and smash of spaces, for example. \\

\indent \textbf{Notation and Terminology}\quad (1) We fix a base ring $\base$. All complexes are defined over $\base$ and considered as cohomologically graded, i.e., differentials raise degree. For usually homologically graded complexes such as operads, we implicitly regard them as cohomologically graded by negating degree. For an element $x$ of a complex, $|x|$ denotes its degree and we put $||x||=|x|-1$.\\
\indent A subset of an ordered set is always considered as an ordered set with the induced order. For an integer $k\geq 1$, $\bar{k}$ denotes the ordered set $\{1,\dots, k\}$ with the usual order.\\
\indent Let $S$ be a finite set. $|S|$ denotes the cardinality of $S$. $||S||$ denotes $|S|-1$ if $S$ is non-empty, $0$ otherwise.\\

\indent (2) For $k\geq 0$, $\Sigma_k$ denotes the $k$-th symmetric group
. As usual, an operad $\mathcal{O}=\{\oper (k)\}_{k\geq 0}$ is a sequence of $\Sigma_k$-modules $\oper (k)$ equipped with composition multiplications which satisfy associativity and equivariance (see \cite{krizmay}). We denote by $\bar{\oper}$ the sequence obtained from $\oper$ by replacing $\oper (0)$ with the zero module. $\bar{\oper}$ has natural operad structure induced from $\oper$.   \\
\indent The sequence operad is defined in \cite{mccluresmith}. The surjection operad in \cite{bergerfresse} is the same thing except for sign difference. By definition, its $k$-th module is the free graded abelian group generated by non-degenerate sequences $f:\bar{m}\to \bar{k}$, whose homological degree is $m-k$. 
As in \cite{mccluresmith}, a non-degenerate sequence $f$ is presented as $(f(1)f(2)\dots f(m))$. For example, $(12)$ denotes the identity on $\bar{2}$. We denote by $\E$ the sequence operad tensored with $\base$. We entirely follow the sign rules of \cite{mccluresmith}. For differential $d$ ($\partial$ in \cite{mccluresmith}), we write that
\[
df=\sum_{q=1}^m(-1)^{\tau' _f(q)}d_qf, 
\]
where $d_qf=f_{\bar{m}-\{q\}}$ and $\tau'_f(q)=\tau _f(q)-f(q)$ in the notation of \cite{mccluresmith}. To simplify notations we put $f\cdot g:=(12)(f,g)$ for $f,g\in\E$. We denote by $\diamond$ the action of the symmetric groups, like $f\diamond \sigma$ for $f\in\E(k)$ and $\sigma\in \Sigma_k$. We omit $\cdot$ and $\diamond$ if it does not cause confusion. \\
\indent We use chain maps $r_a:\E(k)\to \E(k-1)$, $\iota_a:\E(k-1)\to \E(k)$ and a chain homotopy $s_a:\E(k)\to\E(k)$ for $a\in\bar{k}$. For a non-degenerate sequence $f:\bar{m}\to \bar{k}$, $r_a(f)=0$ if $|f^{-1}(a)|\geq 2$, or otherwise, $r_a$ removes $a$ from $f$ and decreases values bigger than $a$ by 1. $\iota_a$ places $a$ at the begining of $f$ and increases values of $f$ bigger than or equal to $a$ by 1. $s_a$ places $a$ at the begining of $f$ and multiply $(-1)^{\sum_{i<a}||f^{-1}(i)||}$. $r_a$, $\iota_a$, and $s_a$ satisfy
\[
ds_a+s_ad=id+\iota _ar_a.
\]
These chain maps and chain homotopy are variations of those defined in \cite{mccluresmith}. \\
\indent $\E$ naturally acts on the normalized cochain complex $N^*(X,\base)$ of a simplicial set $X$ (see \cite{mccluresmith}). If $X$ is pointed, the reduced normalized cochain $\bar{N}^*(X,\base )$ has induced action of $\bar{\E}$.\\
\indent Let $\A$ denote the associative operad. we always identify $\A$ with the suboperad of $\E$ consisting of degree zero elements.\\

\indent (3) Let $A$ be a (not necessarily unital) associative dg-algebra. The bar complex $BA$ of $A$ is defined as follows. The module of degree $d$ is given by 
\[
B^dA=\bigoplus_{l\geq 1}\bigoplus _{d_1+\cdot +d_l-l=d}A^{d_1}\otimes_\base\cdots\otimes_\base A^{d_l}
\]
A tensor $a_1\otimes\cdots\otimes a_l$ considered as an element of $BA$ is denoted by $[a_1|\cdots|a_l]$ so that $|[a_1|\cdots|a_l]|=\sum_{j=1}^l||a_j||$. We call $l$ the length of $[a_1|\cdots|a_l]$. 
The differential $d$ is the sum of two differentials $d=d_0+d_1$, where
\[
\begin{split}
d_0([a_1|\cdots |a_l])=&\sum_{j=1}^l(-1)^{\sum_{k<j}||a_k||}[a_1|\cdots |da_j|\cdots |a_l]\\
d_1([a_1|\cdots |a_l])=&\sum_{j=1}^{l-1}(-1)^{\sum_{k\leq j}||a_k||}[a_1|\cdots |a_ja_{j+1}|\cdot a_l]
\end{split}
\]
For an $\bar\E$-algebra $A$, we may regard $A$ as an associative algebra by forgetting structures and form the bar complex $BA$.
\section{The chain map $\Phi_k$}
In this section, we define the chain map $\Phi_k$ in Thm.\ref{mainthm} and show some properties of it. Let $A$ be an $\bar{\E}$-algebra. In the following, we write $f(\mbf{x}^1,\dots, \mbf{x}^k):=\Phi_k(f,\mbf{x}^1,\dots, \mbf{x}^k)$ for $f\in\E(k)$ and $\mbf{x}^1,\dots,\mbf{x}^k\in BA$.\\
\indent Let $\mbf{x}^i$ be an element of the form $[x_1^i|\cdots |x_{p^i}^i]$ for each $i=1,\dots,k$ ($x^i_t\in A$, $p^i\geq 1$). Roughly speaking, $f(\mbf{x}^1,\dots, \mbf{x}^k)$ is of the following form.
\[
f(\mbf{x}^1,\dots, \mbf{x}^k)=\sum_{\alpha}\ \big[f^\alpha _1(y^\alpha_{1},\dots y^\alpha_{q_1})\big| f^\alpha _2(y^\alpha_{q_1+1},\dots, y^\alpha_{q_2})\big|\cdots \big|f^\alpha_l  (y^\alpha_{q_{l-1}+1},\dots y^\alpha_{p}) \big]
\] 
Here, the sum runs through a certain set of indices and 
$f^\alpha _1,\dots, f^\alpha _l$ are elements of $\bar \E$ associated to an index $\alpha$, which we call \textit{coefficient elements}. $y_1^\alpha,\dots ,y_{p}^\alpha$ is a permutation of $x^1_1,\dots, x^1_{p^1},x^2_1,\dots, x^2_{p^2},\dots,$ $x^k_1,\dots,x^k_{p^k}$ associated to $\alpha$ ($p=p^1+\cdots +p^k$). (In practice, the permutation is  a $(p^1,\dots, p^k)$-shuffle.)  We first define the set of indices, then define coefficient elements inductively, using the indices.  
\subsection{Indices}\label{index}
Let $p^1,\dots, p^k$ be positive integers.
An \textit{elementary decomposition of the set $\overline{p^1}\sqcup\cdots\sqcup\overline{p^k}$ with $l$-pieces} is $l$-tuple $(E_1,\dots, E_l)$ of non-empty subsets of $\overline{p^1}\sqcup\cdots\sqcup\overline{p^k}$ such that 
\begin{itemize}
\item For each $i\in\bar k$ and each pair $j<j'$ in $\bar{l}$, if $E_j\cap \ol{p^i}$ and $E_{j'}\cap \ol{p^i}$ are non-empty, any element of $E_j\cap \ol{p^i}$ is smaller than any element of $E_{j'}\cap \ol{p^i}$,
\item $E_1\sqcup\cdots\sqcup E_l=\overline{p^1}\sqcup\cdots\sqcup\overline{p^k}$.
\end{itemize}
For example, let $k=2$ and $(p^1,p^2)=(2,3)$. We write $\bar 2=\{a_1<a_2\}$ and $\bar 3=\{b_1<b_2<b_3\}$. An elementary decomposition of $\bar 2\sqcup\bar 3$ with 3 pieces is given by $E_1=\{b_1\}$, $E_2=\{a_1, b_2, b_3\}$, $E_3=\{a_2\}$.\\
\indent Recall from \cite{mccluresmith} the notion of an \textit{overlapping partition}. An overlapping partition of a totally ordered set $T$ with $l$-pieces is an $l$-tuple $(A_1,\dots, A_l)$ of non-empty subsets of $T$ such that the last element of $A_j$ is equal to the first element of $A_{j+1}$ for each $j=1,\dots, l-1$.\\
\indent Let $f:\bar{m}\rightarrow \bar{k}$ be a non-degenerate sequence. A \textit{valuewise overlapping partition of $f$ with $l$ pieces} is an $l$-tuple $(S_1,\dots,S_l)$ consisting of non-empty subsets of $\bar m$ which satisfies the following condition: For each $i\in\bar{k}$, if $j_1<\cdots<j_s$ denote the all $j$'s such that $i\in f(S_j)$, then $(S_{j_1}\cap f^{-1}(i),\dots,S_{j_s}\cap f^{-1}(i))$ is an overlapping partition of $f^{-1}(i)$. Some examples of valuewise overlapping partitions of $(1212)$ with 2 pieces are
\[
(S_1,S_2)=(\fbox{121} 2, 1 \fbox{212}),\  (\fbox{1212}, 12\fbox{12}),\  (\fbox{1}2\fbox{1}2, 1\fbox{2}1\fbox{2}), \ (1\fbox{2}12, \fbox{1212}).
\]
Here, the first example denotes $(\{1,2,3\}, \{2,3,4\})$ for instance.
\begin{defi}
Let $f:\ol{m}\to \ol{k}$ be a non-degenerate sequence. Let $l$ and $p^1,\dots p^k$ be positive integers. An \textup{$l$-index $\alpha=(E_1,\dots, E_l; S_1,\dots, S_l)$ of} $(f;p^1,\dots,p^k)$ consists of 
\begin{itemize}
\item An elementary decomposition $(E_1,\dots, E_l)$ of $\ol{p^1}\sqcup\dots\sqcup\ol{p^k}$ with $l$-pieces and
\item A valuewise overlapping partition $(S_1,\dots,S_l)$ of $f$ with $l$ pieces.
\end{itemize}
such that $f(S_j)=\{i\in\ol{k}|\ol{p^i}\cap E_j\not=\emptyset \}$.\\
\indent The set of all $l$-indices of $(f;p^1,\dots,p^k)$ is denoted by $A_l(f;p^1,\dots, p^k)$ or $A_l(f)$ if there is no danger of confusion.
\end{defi}
We will associate a term of length $l$ to each $l$-index.


\subsection{Coefficient elements} We shall define  a coefficient elemement
\[
C(f;e^1,\dots, e^k)\in \E(e^1+\cdots +e^k)_{e^1+\cdots +e^k+m-k-1}
\]
 \vspace{2pt} for each element $f\in\E(k)_{m-k}$ and each integers $e^1,\dots ,e^k\geq 1$. Here, $\E(r)_d$ denotes the module of homological degree $d$ of $\E(r)$.\\

We use induction to define coefficient elements. We give the lexicographical order on the set \\
$\{(k,m,e^1,\dots, e^k)|k,m,e^1,\dots, e^k \geq 1\}$. 
In other words, $(k',m',s^1,\dots, s^{k'})< (k, m,e^1,\dots, e^k)$ if
\[
k'<k \text{ or } (k'=k \text{ and } m'<m)\\
 \text{ or } (k'=k, m'=m, s^1= e^1,\dots, s^{i-1}=e^{i-1}, s^i<e^i \text{ for some $i\in\bar k$ }).
\]
We put $C(f,1)=(1)\in\E(1)_0$ for $f=(1)\in\E(1)_0$. 
 Suppose $C(f';s^1,\dots, s^{k'})$ is defined for $f'\in \E(k')_{m'-k'}$ and $s^1,\dots,s^{k'}\geq 1$ such that 
$(k',m',s^1,\dots, s^{k'})<(k,m,e^1,\dots,e^k)$.\\
\indent Let $f:\ol{m}\to \ol{k}$ be a non-degenerate sequence. We shall introduce some notations. Let $\alpha=(E_j;S_j)\in A_l(f;e^1,\dots, e^k)$ be an $l$-index. For each $j=1,\dots, l$, and $i=1,\dots k$, we put 
\[
e^i(E_j)=|\ol{p^i}\cap E_j|,\quad e(E_j)=\sum_{i=1}^ke^i(E_j).
\] 
Let $f_{S_j}$ denote the composition:
\[
\ol{|S_j|}\cong S_j\stackrel{f|_{S_j}}{\to} f(S_j)\cong \ol{|f(S_j)|},
\]
where $\ol{|S_j|}\cong S_j$ and $f(S_j)\cong \ol{|f(S_j)|}$ are the order-preserving bijections. If $f_{S_j}$ is non-degenerate as a sequence, we put
\[
C(f_{S_j};E_j)=C(f_{S_j};e^{i_1}(E_j),\dots, e^{i_t}(E_j)),
\]
where $i_s$'s are integers such that $f(S_j)=\{i_1<\dots <i_t\}$, otherwise, we put $C(f_{S_j};E_j)=0$\\
\indent Let $\phi _{e^1,\dots, e^k}:\ol{e^1}\sqcup\cdots\sqcup\ol{e^k}\to \ol{e^1+\cdots +e^k}$ be the bijection given by $\ol{e^i}\ni r\mapsto e^1+\cdots +e^{i-1}+r$. We give $\ol{e^1}\sqcup\cdots\sqcup\ol{e^k}$ a total order  such that $\phi _{e^1,\dots, e^k}$ preserves the order. Note that an elementary decomposition $(E_1,\dots, E_l)$ of $\ol{e^1}\sqcup\cdots\sqcup\ol{e^k}$ defines a map $\mu=\mu_{E_1,\dots, E_l}:\ol{e^1}\sqcup\cdots\sqcup\ol{e^k}\to  \ol{e^1+\cdots +e^k}$ such that the restriction $\mu|_{E_j}$ is order-preserving for each $j$ and any element of $\mu(E_j)$ is smaller than any element of $\mu(E_{j'})$ for each $j<j'$. ($\mu$ is a $(e^1,\dots,e^k)$-shuffle.) For an $l$-index $\alpha=(E_j,S_j)$, $\sigma_\alpha\in \Sigma_{e^1+\cdots +e^k}$ denotes the permutation given by the composition
\[
\xymatrix{\ol{e^1+\cdots +e^k}\ar[rr]^{\phi_{e^1,\dots ,e^k}^{-1}}&&\ol{e^1}\sqcup\cdots\sqcup\ol{e^k}\ar[rr]^{\mu_{E_1,\dots, E_l}}&&\ol{e^1+\cdots +e^k}.}
\]
\indent We put 
\[
\begin{split}
X_1(f;e^1,\dots, e^k)&=C(df;e^1,\dots e^k),\\
X_2(f;e^1,\dots, e^k)&=\sum_\alpha (-1)^{\theta (f,\alpha)}(C(f_{S_1};E_1)\cdot C(f_{S_2};E_2))\diamond \sigma_\alpha .
\end{split}
\]
Here, the sum runs through all 2-indices $\alpha=(E_1,E_2;S_1,S_2)\in A_2(f;e^1,\dots, e^k)$, and 
\[
\theta(f,\alpha)=1+|f_{S_1}|+e(E_1)(|f_{S_2}|+e(E_2)-1)+\sum_{i'<i}e^i(E_1)\cdot e^{i'}(E_2)+\sum_{i'<i}||f_{S_1}^{-1}(i)||\cdot ||f_{S_2}^{-1}(i')||.
\]
Finally we put
\[
C(f;e^1,\dots, e^k)=s_a(X_1(f;e^1,\dots, e^k)+X_2(f;e^1,\dots, e^k)),
\]
where $a=\sum_{i<f(1)}e^{i}+1$ (see Notation and Terminology (2)). If $(S_1,\dots, S_l)$ is a valuewise overlapping partition of $f$ such that $f_{S_j}$ is non-degenerate for each $j$, $|f|=|f_{S_1}|+\cdots+|f_{S_l}|$. So the element defined above actually has the prescribed degree. For a general element $f\in\E(k)_{m-k}$, we define the coefficient element by extending the above definition linearly.
\begin{exa}
\textup{(1)} $C((12);1,q)=(-1)^{q(q+3)/2}(12131\dots, q+1,1)$, $C((12);p,q)=0$ for $p>1$. \\
\textup{(2)} $C((123);1,1,1)= -(12131)-(13121)+(12321)$, $C((123);1,2,1)=-(1232141)-(1213141)-(1214131)+(1213431)-(1412131)$. (For further examples, see Prop.\ref{propcomparison}.)
\end{exa}


\subsection{Definition of $\Phi_k$}
Let $A$ be an $\bar{\E}$-algebra. Let $\mathbf{x}^1=[x^1_1|\cdots |x^1_{p^1}],\dots, \mathbf{x}^k=[x^k_1|\cdots |x^k_{p^k}]$ be elements of $BA$ of length $p^1,\dots,p^k$, respectively. Let $\alpha=(E_1,\cdots,E_l;S_1,\cdots, S_l)\in A_l(f;p^1,\dots p^k)$ be an $l$-index. 
Let $E_j\mbf{x}$ denote the "substitution" of $\mbf{x}^1$,$\dots$,$\mbf{x}^k$ to $E_j$. In other words, $E_j\mbf{x}$ is a $e(E_j)$-tuple of variables such that $x^i_t$ occurs in $E_j\mbf{x}$ if and only if $t\in\ol{p^i}$ belongs to $E_j$ and the order of occurrences of variables is the same as the order occurrences of corresponding elements of $E_j$. For example, if $(E_1,E_2,E_3)=(\{b_1\},\{a_1,b_2,b_3\},\{a_2\})$, $E_1\mbf{x}=(x^2_1)$, $E_2\mbf{x}=(x^1_1,x^2_2,x^2_3)$, and $E_3\mbf{x}=(x^1_2)$.  We put 
\[
f(\mbf{x}^1,\dots, \mbf{x}^k)=\sum_{l=1}^{p}\sum_{\alpha}(-1)^{\kappa (f,\alpha,\mbf{x})}\big[ C(f_{S_1};E_1)(E_1\mbf{x})\big|\cdots 
\big|C(f_{S_l};E_l)(E_l\mbf{x})\big]\in B(A).
\]
Here, $p=p^1+\cdots+p^k$ and $\alpha=(E_1,\dots, E_l;S_1,\dots,S_l)$ runs through all $l$-indices in $A_l(f;p^1,\dots, p^k)$, and 
\[
\begin{split}
\kappa (f,\alpha,\mbf{x})=\sum_{\scriptsize 
\begin{array}{c}
(i,j),(i',j')\\
i>i'\\
j<j'
\end{array}}&\Big\{||f_{S_j}^{-1}(i)||\cdot ||f_{S_{j'}}^{-1}(i')||+||E_j\mbf{x}^i||\cdot ||E_{j'}\mbf{x}^{i'}||\Big\}\\ 
+
\sum_{j'<j}|f_{S_j}|&\cdot ||E_{j'}(\mbf{x})||+\sum_{j=1}^l\Big\{\sum_{i
\leq i'}e^{i'}(E_j)\cdot ||E_j\mbf{x}^i||+\sum_{i=1}^k\sum_{t=1}^{e^i(E_j)}t\cdot ||E_j\mbf{x}^i_t||\Big\},
\end{split}
\]
where $E_j\mbf{x}^i_t$ is the variable corresponding to the $t$-th element of $E_j\cap \ol{p^i}$, i.e., $E_j\mbf{x}^i_t=x^i_T$, $T=e^1(E_j)+\cdots e^{i-1}(E_j)+t$, and 
\[
||E_j\mbf{x}^i||=\sum_{t=1}^{e^i(E_j)}||E_j\mbf{x}^i_t||,\quad ||E_j\mbf{x}||=\sum_{i=1}^k||E_j\mbf{x}^i||.
\]
\indent For general elements of $\E (k)$, we extend the above definition linearly and  obtain a map 
$\Phi_k: \E (k)\otimes_\base BA^{\otimes k}\to BA$.
\begin{exa}Let $\base$ be a field of characteristic 2.\\
\textup{(1)} $(12)([x],[y])=[x|y]+[y|x]+[(121)(x,y)]$.\\
\textup{(2)} $(12)([x],[y_1|y_2])=[x|y_1|y_2]+[y_1|x|y_2]+[y_1|y_2|x]+[(121)(x,y_1)|y_2]+[y_1|(121)(x,y_2)]+[(12131)(x,y_1,y_2)]$.
\end{exa}


\subsection{Compatiblity with differentials and invariance under the action of $\Sigma_k$}
We shall show $\Phi_k$ is a chain map and invariant under the action of $\Sigma_k$. We first set 
\[
X_3(f;e^1,\dots ,e^k)=\sum_{i=1}^{k}\sum_{t=1}^{e^i-1}(-1)^{\lambda (f,i,t,e^1,\dots, e^k)}C(f;e^1,\dots, e^i-1,\dots, e^k)((1),\dots, (12),\dots ,(1)).
\]
Here, $((1),\dots, (12),\dots, (1))$ denotes the $(\sum_{i=1}^{k} e^i-1)$-tuple such that its $(e^1+\cdots+e^{i-1}+t)$-th component is the identity $(12)\in\E(2)_0$ on $\bar{2}$ and others are $(1)$'s $\in \E(1)_0$, and 
\[
\lambda (f,i,t,e^1,\dots, e^k)=1+t+|f|+\sum_{i'\geq i}e^{i'}.
\]
\begin{lem}\label{lemchain}
\textup{(1)} For each non-degenerate sequence $f:\bar{m}\to \bar{k}$ and each positive integers $e^1,\dots e^k$,
\[
dC(f;e^1,\dots, e^k)=X_1(f;e^1,\dots, e^k )+X_2(f;e^1,\dots, e^k)+X_3(f;e^1,\dots , e^k).
\]
\textup{(2)} If $k\geq 2$, $C(f;e^1,\dots, e^k)$ is a linear combination of non-degenerate sequences $g$ which satisfy the following conditions:
\begin{enumerate}
\item Each of $g(1)$ and $g(2(\sum_{i=1}^k e^i)+m-k-1)$ occurs  at least two times in the 
sequence $g$.
\item If $b\in \ol{\sum_{i=1}^k e^i}$ occurs exactly one time in $g$, both sides of $b$ are occupied by the same number. In other words, $g$ is of the form $(\cdots cbc \cdots)$.
\end{enumerate}
\end{lem}
\begin{proof}
We use induction. Suppose that (1) and (2) hold for $f'\in \E(k')_{m'-k'}$ and $s^1,\dots,s^{k'}\geq 0$ such that $(k',m',s^1,\dots, s^{k'})<(k,m,e^1,\dots,e^k)$ in the lexicographical order.
To show the former statement of the lemma, it is enough to show $d(X_1+X_2+X_3)=0$ and $\iota_ar_a(X_1+X_2+X_3)=0$ for $a=\Sigma_{i<f(1)}e^1+1$ (see Notation and Terminology (2)).\\
\indent We shall show $d (X_1+X_2+X_3)=0$. Using inductive hypothesis for (1), we expand $d(X_1+X_2+X_3)$ as follows.
\[
d(X_1+X_2+X_3)=\sum_{n, n'}(X_n\to X_{n'})
\]
Here, $X_n\to X_{n'}$ is the term obtained by replacing coefficient elements $C(g;\cdots)$ appearing in $X_n$ by $X_{n'}(g;\cdots )$ for example, $X_3\to X_2$ denotes the term $\sum \pm X_2(f;e^1,\dots, e^j-1,\dots, e^k)(\dots, (12),\dots)$. It is easy to see $X_1\to X_1=X_2 \to X_2=X_3\to X_3=(X_2\to X_3)+(X_3\to X_2)=(X_1\to X_3)+(X_3\to X_1)=0$. We shall show $
(X_1\to X_2)+(X_2\to X_1)=0$.
\[
X_1\to X_2=
\sum_{q=1}^m\sum _{\alpha' =(E'_j;S'_j)\in A_2(d_qf)}(-1)^{\tau '_q(f)+\theta (d_qf,\alpha ^q)}(C((d_q f)_{S_1'};E_1')\cdot C((d_qf)_{S_2'};E_2'))\diamond \sigma_{\alpha'}
\]
\[
\begin{split}
X_2\to X_1=
\sum_{\alpha =(E_j;S_j)\in A_2(f)}
  \Big\{ & \sum _q(-1)^{\theta (f,\alpha )+\tau '_q(f_{S_1})}C(d_q  (f_{S_1});E_1)\cdot C(f_{S_2};E_2)\\
+&\sum _q(-1)^{\theta (f,\alpha )+\tau '_q(f_{S_2})+|C(f_{S_1};E_1)|}C(f_{S_1};E_1)\cdot C(d_q(f_{S_2});E_2)\Big\}
\diamond \sigma_{\alpha }.
\end{split}
\]
So the terms of $X_1\to X_2$ is indexed by the set $R$ of pairs $(q,(E_j',S_j')_{j=1,2})$ consisting of an integer $q\in \bar{m}$ and 2-index $(E_j',S_j')\in A_2(d_q f;e^1,\dots e^k)$ such that $f(q)$ occurs at least two times in $f$, while the terms of $X_2\to X_1$ is indexed by the set $T_1\sqcup T_2$ where $T_s$ is the set of pairs $((E_j,S_j)_{j=1,2},q)$ consisting of 2-index of $f$ and an integer $q\in S_s$ such that $f(q)$ occurs at least two times in $f_{S_s}$ ($s=1,2$). We put
\[
T_s'=\{ ((E_j,S_j),q)\in T_s| q\not\in S_{s'}\text{ for }s'\not= s\}, \qquad T_s''=T_s-T_s'.
\]
Define a map $F:T_1'\sqcup T_2'\to R$ by $T_1'\ni((E_j,S_j),q)\mapsto (q,(E_1,E_2;S_1-{q},S_2))$, $T_2'\ni((E_j,S_j),q)\mapsto (q,(E_1,E_2;S_1,S_2-{q}))$. $F$ is well-defined as $((S_1-\{q\})\cap (d_{q}f)^{-1}(i), S_2\cap (d_{q}f)^{-1}(i))$ is an overlapping partition of $(d_{q}f)^{-1}(i)$ and similarly for $T_2'$. The term corresponding to an element of $T_1'\sqcup T_2'$ is equal, up to sign, to the term corresponding to its image by $F$. We shall check the signs.
\[
\begin{split}
\theta(d_qf,\alpha' )-\theta(f,\alpha) &=1+\sum_{i'<f(q)}||f_{S_2}^{-1}(i')|| \\
\tau' _q(f)-\tau'_{q}(f_{S_1}) &=\sum_{i<f(q)}(||f^{-1}(i)||-||f_{S_1}^{-1}(i)||)+|\{q'\not\in S_1| q'< q,\  f(q')=f(q)\} |
\end{split}
\]
As $(S_1\cap f^{-1}(i),S_2\cap f^{-1}(i))$ is an overlapping partition, $||f^{-1}(i)||=||f_{S_1}^{-1}(i)||+||f_{S_2}^{-1}(i)||$ and $|\{q'\not\in S_1| q'< q,\  f(q')=f(q)\} |=0$. So we have $(\theta(d_qf,\alpha_q )+\tau' _q(f))-(\theta(f,\alpha)+\tau_{q}(f_{S_1}))\equiv 1$. Similarly for $T'_2$. ( Verification of signs is a tiresome but straightforward task. In the rest of the paper we leave it to the reader.)
\\
\indent The inverse $F'$ of $F$ is given as follows. Let $(q,(E_j',S_j'))\in R$. $(S_1'\cap(d_qf)^{-1}(f(q)), S_2'\cap(d_qf)^{-1}(f(q)))$ is naturally regarded as an overlapping partition of $f^{-1}(f(q))-{q}$ so we have the following four cases. (i)$S_1'\cap(d_qf)^{-1}(f(q))=\emptyset$, (ii) $S_2'\cap(d_qf)^{-1}(f(q))=\emptyset$, (iii) the overlapping point $u$ is smaller than $q$. (iv) the overlapping point $u$ is bigger than $q$. In the cases (i) and (iii), we put $F'((q,(E_j',S_j')))=((E_1',E_2';S_1',S_2'\cup\{q\}), q)\in T'_2$ and in the cases (ii) and (iv), we put  $F'((q,(E_j',S_j')))=((E_1',E_2';S_1'\cup\{q\},S_2'), q)\in T'_1$. Here, we regard $S_1'$ and $S_2'$ as subsets of $\bar{m}-\{q\}$.\\
\indent We define a map $G:T''_1\to T''_2$. For $((E_1,E_2;S_1,S_2),q)\in T''_1$,  we put $G((E_1,E_2;S_1,S_2),q)=((E_1,E_2;S_1-\{q\},S_2\cup\{q'\}), q')$ where $q'$ is the previous element of $q$ in $f^{-1}(f(q))$. $G$ is well-defined and bijective. This implies the terms indexed by $T''_1$ and the terms indexed by $T''_2$ cancel each other. Thus, we have shown $(X_1\to X_2)+(X_2\to X_1)=0$ and $d(X_1+X_2+X_3)=0$.\\
\indent We shall show that $\iota_ar_a(X_1+X_2+X_3)=0$. We use the inductive hypothesis for (2). We easily see $r_a(X_1)=0$. Among the terms of $X_2$, the ones which are not vanished by $r_a$ are indexed by 2-indices $(E_j,S_j)$ such that $E_1=\{1\}\subset\ol{e^{f(1)}}$ or ($e^{f(1)}=1$ and $E_2=\ol{e^{f(1)}}$). Among the terms of $X_3$, there is a unique diagram $D$ appearing in the composition $(12)\circ_{\Sigma_{i<f(1)}e^i+1}C(f;\dots, e^{f(1)}-1,\dots)$, such that the sequences of the form $h_D$ is not vanished by $r_a$ (see the definition of composition multiplications in \cite[Definition 2.23-2.25, Proposition 2.26]{mccluresmith}) . In the case $e^{f(1)}\geq 2$, the term in $X_2$ and the term in $X_1$ cancel. In the case $e^{f(1)}=1$, we see $\iota_ar_a((\text{the term of }E_1=\ol{e^{f(1)}}) + (\text{ the term of }E_2=\ol{e^{f(1)}}))=0$.\\
\indent (2) of the lemma easily follows from the hypothesis for (2).
\end{proof}

\begin{lem}\label{lemequivariance}
For $\sigma\in\Sigma_k$,
\[
C(f\diamond \sigma ;e^1,\dots, e^k)=(-1)^{\xi(f,\sigma, e^1,\dots,e^k)}C(f;e^{\sigma^{-1}(1)},\dots ,e^{\sigma^{-1}(k)})\diamond\sigma (e^1,\dots, e^k),
\] where $\xi(f,\sigma,e^1,\dots,e^k)=\sum_{i<i'\text{s.t.}\sigma(i)>\sigma(i')}e^i\cdot e^{i'}$ and $\sigma (e^1,\dots, e^k)\in\Sigma_{e^1+\cdots +e^k}$ is the block permutation (see \cite{krizmay}) .
\end{lem}
\begin{proof}
This follows from straightforward induction.
\end{proof}
\begin{thm}\label{thmchain}
For each $k\geq 1$, $\Phi_k$ defines a chain map $\Phi_k:\E(k)\otimes_{\Sigma_k}BA^{\otimes k}\longrightarrow BA$.
\end{thm}
\begin{proof}
In the following, we omit signs for simplicity. Verification of signs is left to the reader. We shall show $\Phi_k$ defines a chain map $\E(k)\otimes_\base BA^{\otimes k}\longrightarrow BA$. We first check the terms of length 1. the terms of length 1 in $d(f(\mbf{x}^1,\dots,\mbf{x}^k))$ is
\[
\begin{split}
dC(f;p^1,\dots,p^k)(x_1^1,\dots, x_{p^k}^k)&+\sum_{\alpha}C(f_{S_1};E_1)(E_1\mbf{x})\cdot C(f_{S_2};E_2)(E_2\mbf{x})\\
&+
\sum_{i,t}C(f;p^1,\dots,p^k)(\dots, dx^i_t,\dots )
\end{split}
\]
On the other hand, the terms of length 1 in $(df)(\mbf{x}^1,\dots, \mbf{x}^k)+fd(\mbf{x^1},\dots,\mbf{x^k})$
is
\[
C(df;p^1,\dots, p^k)(x^1_1,\dots,x^k_{p^k})+\sum_{i,t}C(f;\dots ,p^i-1,\dots)(\dots,x^i_tx^i_{t+1},\dots )+\sum_{i,t} C(f;p^1,\dots, p^k)(\dots, dx^i_t,\dots)
\]
 By Lem.\ref{lemchain}, these terms are equal. Equality of higher length terms essentially follows form equality of length 1. The only non-trivial part is equality of the following two sums.
\[
\begin{split}
&\sum_q\sum_{\alpha\in A_l(d_qf;p^1,\dots, p^k)}[C((d_qf)_{S_1};E_1)(E_1\mbf{x})|\dots |C((d_qf)_{S_q};E_q)(E_q\mbf{x})], \\
&\sum_{\alpha\in A_l(f;p^1,\dots, p^k)}\sum_{j_0}\sum_{q}
[C(f_{S_1};E_1)(E_1\mbf{x})|\dots| C(d_q(f_{S_{j_0}});E_{j_0})(E_{j_0}\mbf{x})|\dots |C(f_{S_l};E_l)(E_l\mbf{x})]
\end{split}
\]
This follows from a consideration similar to the proof of Lem.\ref{lemchain}.  Precisely speaking, the latter sum is indexed by the set $T$ as follows.
\[
T=\{ ((E_{j},S_{j}),j_0,q)|\ (E_{j},S_{j})\in A_l(f;p^1,\dots, p^k),\ 1\leq j_0\leq l,\  q\in S_{j_0}, \ |f^{-1}(f(q))\cap S_{j_0}|\geq 2\}
\] 
$T$ is decomposed to disjoint union of three subsets $T_1$, $T_2$, $T_3$. Here,
\[
\begin{split}
T_1=&\{ ((E_{j},S_{j}),j_0,q)\in T|\forall j'\not= j_0, \ q\not\in S_{j'}\} 
\\
T_2=&\{ ((E_{j},S_{j}),j_0,q)\in T|\exists j'\not= j_0, \ q\in S_{j'}\text{ and }\forall j'<j_0,\  q\not \in S_{j'}\}\\
T_3=&\{ ((E_{j},S_{j}),j_0,q)\in T|\exists j'< j_0,\  q\in S_{j'}\}
\end{split}
\]
We can see elements of $T_1$ are in one to one correspondence with indices of the former sum and terms indexed by $T_2$ and $T_3$ cancel each other. \\
\indent The claim that the chain map factors through the coinvariant follows easily from Lem.\ref{lemequivariance}. 
\end{proof}


\subsection{Properties of $\Phi_k$}

Recall the sequence operad has a filtration $\{\E_n\}_{n\geq 0}$,where $\E_n$ is weak equivalent to the chain operad of little $n$-cubes and $\E_n$ is spanned by sequences of complexity $\leq n$ (\cite{mccluresmith}).
\begin{prop}\label{propcomparison}
\textup{(1)} If $\base$ is a field of characteristic 2, Kadeishvili's $\cup_i$-product on the bar complex \cite{kadeishvili} is equal to evaluation of $\Phi_2$ at $\tau_i:=(1212\dots)\in \E(2)_i$. In other words, $C(\tau_i;p,q)=E^{i}_{p,q}$ in the notation of \cite{kadeishvili}\\
\textup{(2)} The restriction of $\Phi_k$ to $\E_n$ is determined by the action of $\E_{n+1}$ on $A$. More precisely, for a non-degenerate sequence $f$, any coefficient element $C(f;e^1,\dots,e^k)$ of $f$ is a linear combination of non-degenerate sequences $g$ which satisfy the following conditions:
\begin{enumerate}
\item $g|_{g^{-1}(\overline{e^i})}$ is order-preserving
\item For $i_1\not= i_2$, the complexity of $g|_{g^{-1}(\overline{e^{i_1}}\sqcup\overline{e^{i_2}})}$ is smaller than or equal to the complexity of $f|_{f^{-1}(\{i_1,i_2\})}$ plus one.
\end{enumerate}
Here, we use the identification $\overline{e^1}\sqcup\dots\sqcup\overline{e^k}=\overline{e^1+\cdots +e^k};\  \overline{e^i}\ni r\mapsto \sum_{i'<i}e^{i'}+r$.
\end{prop}
\begin{proof}
This follows from straightforward induction.
\end{proof}
\begin{rem}
If one considers an algebra $A$ over the chain operad of little cubes, the method of \cite{vogt} shows the $E_n$-structure of the bar complex $BA$ defined in \cite{fresse1} is determined by the $E_{n+1}$-structure of $A$.
\end{rem}
$\Phi_k$ does not define an operad action on $BA$, i.e., it does not satisfy the associativity low. For example, $(121)((12)([x],[y]),[z])\not= \{(12131)+(13121)\}([x],[y],[z])$. \\
\indent We shall show some partial associativity of $\Phi_k$. Let $C_q:=C((12);1,q)$.
\begin{prop}\label{propasso}
\textup{(1)} Let $f:\bar{k}\to \bar{k}$ and $g:\bar{n}\to \bar{r}$ be two non-degenerate sequences. Let $p^1,\dots, p^k$ and $q^1,\dots, q^r$ be positive integers. Then,
\[
\begin{split}
C(f\cdot g, p^1,&\dots, p^k, q^1,\dots q^r)\\
&=
\sum_{l\geq 1}\sum_{\alpha\in A_l(g;q^1,\dots ,q^r)}(-1)^{\varpi (g,p,\alpha) }C_l(C(f;p^1,\dots p^k),C(g_{S_1};E_1),\dots ,C(g_{S_l};E_l))\diamond  \hat\sigma_\alpha
\end{split}
\]
where 
\[
\begin{split}
\varpi (g,p,\alpha)=\sum_{i>i', j<j'}||g_{S_j}^{-1}(i)||\cdot &||g_{S_{j'}}^{-1}(i')||+\sum_{i\geq i', j>j'}e^i(E_j)\cdot e^{i'}(E_{j'})\\
&+\sum_j\big\{(l-j)e(E_j)+|g_{S_j}|\cdot (j+\sum_{j'<j}e(E_{j'})) \big\}
\end{split}
\]
and $\hat \sigma_\alpha=id\times \sigma_\alpha$ ($id\in \Sigma_{p^1+\cdots +p^k}$).\\
\textup{(2)} If $|f|=0$, then
$(f\cdot g)(\mbf{x}^1,\dots, \mbf{x}^{k+r})=(-1)^{|g|(|\mbf{x}^1|+\cdots+|\mbf{x}^k|)}f(\mbf{x}^1,\dots,\mbf{x}^k)\cdot g(\mbf{x}^{k+1},\dots,\mbf{x}^{k+r})$. In particular, the restriction of $\Phi_k$ to the associative operad $\Phi_k:\A(k)\otimes_{\Sigma_k}BA^{\otimes k}\longrightarrow BA$ defines an action of the operad $\ol{\A}$.
\end{prop}

\begin{proof}
(1)We use induction. When $f=g=(1)$, the claim is trivial. We assume the claim holds for $f':\bar{k'}\to \bar {k'}$ and $g':\bar{n'}\to \bar{r'}$ such that $(k',r',n')< (k,r,n)$ in the lexicographical order.By Thm.\ref{thmchain}, we may assume $f(1)=1$.\\
\indent According to the definition of the composition multiplication, we have
\[
C_l(C(f;p^1,\dots, p^k),C(g_{S_1};E_1),\dots, C(g_{S_l};E_l))\diamond \hat\sigma_\alpha =\sum _Dh(l,\alpha, D)\diamond\hat\sigma_\alpha.
\]
Here, $h(l,\alpha, D)$ is a linear combination of sequences of the form $h_D$ produced from sequences appearing in $C(f;p^1,\dots, p^k),C(g_{S_1};E_1),\dots, C(g_{S_l};E_l)$, by using the diagram $D$ of type $((121\dots l+1,1),d_1,\dots,d_l)$ as follows.
\[
\xymatrix{\overline{2\Sigma_ip^i -1+\Sigma_{j=1}^ld_j+l}\ar[r]^{b_D}\ar[d]^{a_D}&\overline{2\Sigma_ip^i -1}\sqcup \overline{d_1}\sqcup \cdots \sqcup \overline{d_l}\ar[d]\\
\overline{2l+1}\ar[r]^{(121\dots l+11)}&\overline{l+1}.}
\]
Here, $d_j$ is the number of the entries of $C(g_{S_j};E_j)$, $d_j=2e(E_j)+|S_j|-|g(S_j)|-1$ (see \cite[Definition 2.23-2.25, Proposition 2.26]{mccluresmith}).\\
According to the expansion $C(f;p^1,\dots, p^k)=\sum_{\beta=(T_j,F_j)\in A_2(f)}s_1\{C(f_{T_1};F_1)C(f_{T_2};F_2))\sigma_\beta\}$, \\
  $h(l,\alpha, D)$ is expanded as
\[
h(l,\alpha, D)=\sum_\beta h(l,\alpha,\beta, D).
\]
Thus, the right hand side of the equation of the lemma is the sum indexed by the set $\{(l,\alpha, \beta, D)\}$.\\
\indent On the other hand, the left hand side is equal to 
\[
\sum_{\gamma =(G_j,U_j)\in A_2(f\cdot g)
}
s_1(C((f\cdot g)_{U_1};G_1)C((f\cdot g)_{U_2};G_2)\sigma_\gamma).
\]
Here $\gamma$ runs through all 2-indices. We regard $\bar k$ (resp. $\bar r$) as a subset of $\overline{k+r}$ by identifying it with the set of former $k$ elements (resp. latter $r$ elements). Note that if $\gamma$ corresponds to non-zero term, the first element of $\bar k$ is contained in $U_2$. We have five cases:\\
(i) $U_1\subset \bar k$, \\
(ii) $U_1\subset \bar r$ and $U_2\not\subset \bar k$, \\
(iii) $U_1=\bar r$, and $U_2=\bar k$, \\
(iv) $U_1\not\subset \bar k$, $\bar r$, and $U_2\not\subset \bar k$ \\
(v) $U_1\not\subset \bar k$, $\bar r$ and $U_2\subset \bar k$.\\
\indent We shall consider the case (iv). Using inductive hypothesis, we have
\[
s_1(C((f\cdot g)_{U_1};G_1)C((f\cdot g)_{U_2};G_2)\sigma_\gamma)=
\sum_{(l_1,\alpha_1)}\sum_{(l_2,\alpha_2)}s_1\{(h(l_1,\alpha_1,D_1)\hat\sigma_{\alpha_1})\cdot (h(l_2,\alpha _2,D_2)\hat\sigma_{\alpha _2})\}\sigma_\gamma \\
\]
We put $l=l_1+l_2$, and denote by $\alpha$ the $l$-index made by connecting $\alpha_1$ and $\alpha_2$, and put  $\beta=(G_1\cap\mbf{p},G_2\cap\mbf{p};U_1\cap\bar{k},U_2\cap\bar{k})$, where $\mbf{p}=\ol{p^1}\sqcup\dots \sqcup \ol{p^k}$. We denote by $D$ the diagram made by connecting $D_1$ and $D_2$.
Then, 
$s_1\{(h(l_1,\alpha_1,D_1)\hat\sigma_{\alpha_1})\cdot (h(l_2,\alpha _2,D_2)\hat\sigma_{\alpha _2})\}\sigma_\gamma =h(l,\alpha,\beta,D)\hat\sigma_\alpha$. If $\gamma$ runs through the range of (iv), and $l_1, l_2, \alpha_1, \alpha_2$ varies, the index $(l,\alpha,\beta,D)$ runs through the range of 
\[
\begin{split}
l\geq 2,\ \  \alpha, \ &\beta \text{ :arbitrary, } \\ D:\ {a_D}^{-1}&(1)\supset \{1,2\}, \ \  \exists l'<l \ \  a_D^{-1}(2l'+1)\supset \big\{2e(F_1)+\sum_{j\leq l'}d_j,\  2e(F_2)+\sum_{j\leq l'}d_j+1\big\}.
\end{split}
\] 
Similar consideration shows other cases corresponds to the set of $(l,\alpha,\beta, D)$'s satisfying : \\
(i) $l\geq 1$ and $a_D^{-1}(1)\supset \{1,\dots, 2\Sigma_ie^i(F_1)+1\}$.\\
(ii) $l\geq 2$ and $a_D^{-1}(1)=\{1\}$.\\
(iii) $l\geq 1$ and $a_D^{-1}(1)=\{1\}$.\\
(v) $l\geq 1$ and $a_D^{-1}(2l+1)\supset \{2\Sigma_ie^i(F_1)+\sum_jd_j, 2\Sigma_ie^i(F_1)+\sum_jd_j+1\}$.\qquad
(In any case , $\alpha$ and $\beta$ are arbitrary.)\\
\indent (i)-(v) cover all indices and have no overlap. Thus equality of the lemma holds.\\
\indent (2) follows from (1).
\end{proof}


\section{Steenrod operation on the bar complex}\label{steenrod}
Let $p$ be a prime number. In this section, we assume $\base=\Z/p$. Let $\Pi_p$ be the subgroup of $\Sigma_p$ consisting of all cyclic permutations. 
Let $W$ be the $\Pi_p$-projective resolution of $\Z/p$ defined in \cite[Definition 1.2]{may}.
Fix a $\Pi_p$-equivariant chain map
\[
\varphi :W\longrightarrow \E(p)
\]
which takes the fixed generator $e_0$ to $id\in \E(p)_0$. One can easily construct such a map using the contracting chain homotopy $H=\sum_{k}(-1)^k(\iota_1)^k\circ s_1\circ (r_1)^k:\E(p)\to \E(p)$ (see Notation and Terminology (2)).\\
\indent If $B$ is a complex with chain maps $\phi:\E(p)\otimes_{\Sigma_p}B^{\otimes p}\longrightarrow B$, 
we define operations $P^s$ and $\beta P^s$ by applying the definition in \cite[Definition 2.2 and section 5]{may} to the composition
\[
W\otimes B^{\otimes p}\longrightarrow \E(p)\otimes B^{\otimes p}\longrightarrow B.
\]
In this way, using $\Phi_p$ defined in the previous section, we define operations $P^s$ and $\beta P^s$ on the cohomology of the bar complex $BA$ of $\bar{\E}$-algebra $A$.\\
\indent We do not prove that $H^*(BA)$ is a module over the generalized Steenrod algebra with the above operations for general $\bar{\E}$-algebra $A$.  We prove only the following. (See Notation and Terminology (3).) 
 \begin{thm}
Let $X$ be a simply connected pointed simplicial set of finite type. Then there exists a natural isomorphism which commutes with product and the operations $P^s$ and $\beta P^s$.
\[
H^*(B\bar{N}^*(X,\Z /p))\cong \bar{H}^*(\Omega X, \Z /p)
\] 
\end{thm}
\begin{proof}
We first define an operad $\oper$.
Let $M_p$ denotes the $\Sigma_*$-module as follows.
\[
M_p(k)=\left\{
\begin{array}{ll}
\E (p)& \text{ if }k=p \\
\E (2)_0   & \text{ if }k=2\\
0& \text{ otherwise }
\end{array}
\right.
\]
Let $\oper$ be the free (dg-)operad over $M_p$. Note that $\Phi_p$ and $\Phi_2$ defines $\oper$-algebra structure on $BA$, and similarly 
  $\ol{\E}$-algebras have functorial $\oper$-algebra strucures. Let $\mathcal{O}'\to \mathcal{O}$ be a cofibrant replacement of $\mathcal{O}$. By using a fixed section $\mathcal{O}(p)\to \mathcal{O}'(p)$, we can define the Steenrod operations for $\mathcal{O}'$-algbras  \\
\indent The rest of the proof is completely analogous to the proof of topological validity of $E_\infty$-structure in \cite{fresse1}.
For reader's convenience, we sketch the proof. Let $B_{\bar{\E}}$ (resp. $C_{\bar{\E}}$, resp.$B_{\bar{C}}$) denotes bar module of $\bar{\E}$ (resp. categorical bar module of $\bar{\E}$, resp. bar module of the (non-unital) commutative operad $\bar{C}$). These are the (classical or categorical) bar complexes of operads which are considered as algebras over themselves, see \cite{fresse1}. Using $\Phi_k$ defined in the previous section, we give $B_{\bar{\E}}$ a structure of  $\mathcal{O}$-( or $\mathcal{O}'$-)algebra in right $\bar{\E}$-modules. We also consider $B_{\bar{C}}$ and $C_{\bar{\E}}$ as $\mathcal{O}'$-algebras by pulling back the action of $\bar{\E}$. There exist weak equivalences of $\mathcal{O}'$-algebras in right $\bar{\E}$-modules:
\[
B_{\bar{\E}} \longrightarrow B_{\bar{C}} \longleftarrow C_{\bar{\E}}.
\]
From this diagram, using the model category structure on the category of $\mathcal{O}'$-algebras (in right $\bar{\E}$-modules), we obtain weak equivalences:
\[
B_{\bar{\E}}\longleftarrow X\longrightarrow C_{\bar{\E}},
\]
such that $X$ is cofibrant. Then, by homotopy invariance property of the functor $S_{\bar{\E}}(-,-)$ (see 0.2.1 and 0.3.1 of \cite{fresse1}), we have weak equivalences of $\mathcal{O}'$-algebras (in complexes):
\[
\begin{split}
B\bar{N}^*(X)\simeq B(Q(\bar{N}^*(X)))&=S_{\bar{\E}} (B_{\bar{\E}}, Q(\bar{N}^*(X)))\\
&\simeq S_{\bar{\E}} (X, Q(\bar{N}^*(X)))\simeq S_{\bar{\E}} (C_{\bar{\E}}, Q(\bar{N}^*(X)))\simeq \bar{N}^*(\Omega X),
\end{split}
\]
where $Q(-)$ is a cofibrant replacement.
\end{proof}
\section{Diagonal on $\E$}
Let $\E\otimes\E$ be the operad given by $(\E\otimes \E )(k)= \E(k)\otimes_\base \E(k)$. (Composition and action of $\Sigma_k$ are defined simultaneously.) We shall define a morphism $\Delta :\E\longrightarrow \E\otimes \E$ of operads, which we call a diagonal. For a non-degenerate sequence $f:\bar{m}\rightarrow \bar{k}$, we put
\[
\Delta (f)=\sum_{(S_1,S_2)}(-1)^{\delta (f,S_1,S_2)} f_{S_1}\otimes f_{S_2}
\]
Here, $(S_1,S_2)$ runs through valuewise overlapping partition with 2 pieces (see subsection \ref{index} such that $f(S_1)=f(S_2)=\bar{k}$ and $\delta (f,S_1,S_2)=\sum_{i>i'}||f_{S_1}^{-1}(i)||\cdot ||f_{S_2}^{-1}(i')||$.\\
\indent Compatiblity with differentials follows from a consideration similar to the proof of Lem.\ref{lemchain}. It is easy to see that $\Delta$ commutes with composition multiplications.
\begin{prop}
Let $A$, $B$ be two $\E$-algebras. $A\otimes_\base B$ has an $\E$-algebra structure which satisfies the following conditions:
\begin{enumerate}
\item The algebra structure is functorial in $A$ and $B$.
\item The canonical isomorphisms $A\otimes \base\cong A\cong \base\otimes A$ of complexes preserve the algebra structure 
\end{enumerate}
\end{prop}
\begin{proof}
The action of $\E$ is given by
\[
\E(k)\xrightarrow{\Delta} (\E\otimes \E )(k)\longrightarrow Hom(A^{\otimes k },A)\otimes Hom(B^{\otimes k },B)\longrightarrow Hom((A\otimes B)^{\otimes k}, A\otimes B).
\]
\end{proof}

\textbf{Acknowledgements}\quad The author is grateful to Masana Harada for many valuable comments to improve presentations of the paper.

\clearpage
\textbf{\large Table of signs}\\
 
1.\\

\[
X_2(f;e^1,\dots, e^k)=\sum_\alpha (-1)^{\theta (f,\alpha)}(C(f_{S_1};E_1)\cdot C(f_{S_2};E_2))\diamond \sigma_\alpha
\]
\[
\theta(f,\alpha)=1+|f_{S_1}|+e(E_1)(|f_{S_2}|+e(E_2)-1)+\sum_{i'<i}e^i(E_1)\cdot e^{i'}(E_2)+\sum_{i'<i}||f_{S_1}^{-1}(i)||\cdot ||f_{S_2}^{-1}(i')||
\]
2.\\

\[
X_3(f;e^1,\dots ,e^k)=\sum_{i=1}^{k}\sum_{t=1}^{e^j-1}(-1)^{\lambda (f,i,t,e^1,\dots ,e^k)}C(f;e^1,\dots, e^j-1,\dots, e^k)((1),\dots, (12),\dots ,(1)).
\]
\[
\lambda (f,i,t,e^1,\dots ,e^k)=1+t+|f|+\sum_{i'\geq i}e^{i'}.
\]
3.\\

\[
f(\mbf{x}^1,\dots, \mbf{x}^k)=\sum_{l=1}^{p}\sum_{\quad \alpha \in A_l(f;p^1,\dots p^k)}(-1)^{\kappa (f,\alpha,\mbf{x})}\big[ C(f_{S_1};E_1)(E_1\mbf{x})\big|\cdots 
\big|C(f_{S_l};E_l)(E_l\mbf{x})\big], 
\]
\[
\begin{split}
\kappa (f,\alpha,\mbf{x})=\sum_{((i,j),(i',j'))\in T(f,\alpha)}&\Big\{||f_{S_j}^{-1}(i)||\cdot ||f_{S_{j'}}^{-1}(i')||+||E_j(\mbf{x}^i)||\cdot ||E_{j'}(\mbf{x}^{i'})||\Big\}\\ 
+
\sum_{j'<j}|f_{S_j}|&\cdot ||E_{j'}(\mbf{x})||+\sum_{j=1}^l\Big\{\sum_{i'\geq i}e^{i'}(E_j)\cdot ||E_j(\mbf{x})^i||+\sum_{i=1}^k\sum_{t=1}^{e^i(E_j)}t\cdot ||E_j(\mbf{x})^i_t||\Big\},
\end{split}
\]
4.\\

\[
C(f\diamond \sigma ;e^1,\dots, e^k)=(-1)^{\xi(f,\sigma, e^1,\dots, e^k)}C(f;e^{\sigma^{-1}(1)},\dots ,e^{\sigma^{-1}(k)})\diamond \sigma(e^1,\dots, e^k),
\]
\[
\xi(f,\sigma,e^1,\dots,e^k)=\sum_{i<i'\text{s.t.}\sigma(i)>\sigma(i')}e^i\cdot e^{i'}
\]
5.\\
\[
C(f\cdot g, p^1,\dots, p^k,q^1,\dots q^r)=
\sum_{l\geq 1}\sum_{\alpha\in A_l(g;q^1,\dots ,q^r)}(-1)^{\varpi (g,p,\alpha) }C_l(C(f;p^1,\dots p^k),C(g_{S_1};E_1),\dots ,C(g_{S_l};E_l))\diamond  \hat\sigma_\alpha
\]
\[
\begin{split}
\varpi (g,p,\alpha)=\sum_{i>i', j<j'}||g_{S_j}^{-1}(i)||\cdot &||g_{S_{j'}}^{-1}(i')||+\sum_{i\geq i', j>j'}e^i(E_j)\cdot e^{i'}(E_{j'})\\
&+\sum_j\big\{(l-j)e(E_j)+|g_{S_j}|\cdot (j+\sum_{j'<j}e(E_{j'})) \big\}+(l+p)|g|+p(q-l)
\end{split}
\]
\end{document}